\documentclass[11pt,reqno]{amsart}
\usepackage{amsthm,amsmath,amssymb,mathrsfs}

\theoremstyle{plain}
\newtheorem{theorem}{Theorem} [section]
\newtheorem{corollary}[theorem]{Corollary}
\newtheorem{lemma}[theorem]{Lemma}
\newtheorem{proposition}[theorem]{Proposition}

\theoremstyle{definition}
\newtheorem{definition}[theorem]{Definition}

\newtheorem{remark}[theorem]{Remark}

\def\N{{\mathbb{N}}}
\def\R{{\mathbb{R}}}

\def\F{{\mathcal{F}}}

\def\span{{\text{span}}}

\newcommand{\w}{\omega}
\newcommand{\e}{\varepsilon}

\newcommand{\la}{\lambda}

\newcommand{\dd}{\delta}

\newcommand{\pr}{\textnormal{Pr}}

\newcommand{\rg}{\textnormal{rank}}

\newtheorem{thm*}{Teorema}[section]

\newcommand{\CHI}{\hbox{\raise .4ex \hbox{$\chi$}}}

\def\subset{\subseteq}

\def\F{{\mathcal{F}}}
\def\B{{\mathcal{B}}}
\def\Di{{\mathcal{D}}}

\def\S{{\textbf{S}}}

\def\BPi{{\bold {\Pi}}}

\begin{document}

\title[A DIMENSION REDUCTION SCHEME]
{A dimension reduction scheme for the computation of optimal unions of subspaces}

\author[A. Aldroubi]{A.~Aldroubi}
\address[Akram Aldroubi]{Department of Mathematics\\
Vanderbilt University\\ 1326 Stevenson Center\\ Nashville, TN 37240}
\email[Akram Aldroubi]{akram.aldroubi@vanderbilt.edu}

\author[M. Anastasio]{M.~Anastasio}
\address[M. Anastasio, C. Cabrelli and U. Molter]{Departamento de
Matem\'atica \\ Facultad de Ciencias Exactas y Naturales\\ Universidad
de Buenos Aires\\ Ciudad Universitaria, Pabell\'on I\\ 1428 Capital
Federal\\ Argentina\\ and IMAS, UBA-CONICET, Argentina}
\email[Magal\'{i} Anastasio]{manastas@dm.uba.ar}

\author[C. Cabrelli]{C.~Cabrelli}
\email[Carlos~Cabrelli]{cabrelli@dm.uba.ar}

\author[U.Molter]{U.~Molter}
\email[Ursula~M.~Molter]{umolter@dm.uba.ar}

 \thanks{The research of A.~Aldroubi is supported in part by NSF Grant DMS-0807464. M.~Anastasio, C. Cabrelli and U. Molter are partially supported by Grants UBACyT
X149 and X028 (UBA), PICT 2006-00177 (ANPCyT), and PIP 2008-398 (CONICET).}



\date{\today}
\maketitle

\begin{abstract}

Given a set of points $\F$ in a high dimensional space, the problem of finding a union of subspaces $\cup_i V_i\subset \R^N$ that best explains the data $\F$ increases dramatically with the dimension of $\R^N$. In this article, we study a class of transformations that map the problem into
another one in lower dimension. We use the best model in the low dimensional space to approximate the best solution in the original high dimensional space. We then estimate the error produced between this solution and the  optimal solution in the high dimensional space.

\end{abstract}

\section{Introduction}

Given a set of vectors (points) $\F=\{f_1,\dots,f_m\}$ in  a Hilbert space $\mathcal H$ (finite or infinite dimensional), the problem of finding a union of subspaces $\cup_i V_i\subset \mathcal H$ that best explains the data $\F$ has applications to mathematics and engineering  \cite {CL09,EM09, EV09, Kan01, KM02, LD08, AC06, VMS05}.  The subspaces $V_i$ allowed in the model are often constrained. For example the subspaces $V_i$ may be constrained to belong to a family of closed subspaces $\mathcal C$ \cite {ACM08}.  A typical example for $\mathcal H=\R^N$  is when $\mathcal C$ is the set of subspaces of dimension $k<< N$.  If $\mathcal C$ satisfies the so called Minimum Subspace Approximation Property (MSAP), an optimal solution to the non-linear subspace modeling problem that best fit the data exists, and algorithms to find
these subspaces were developed \cite {ACM08}. Necessary and sufficient conditions for $\mathcal C$ to satisfy the MSAP are obtained in \cite{AT10}.

In some applications the  model is a finite union of subspaces and  $\mathcal H$ is finite dimensional. Once the model is found, the given data points can be clustered and classified according to their distances from the subspaces, giving rise to the  so called {\em subspace clustering problem} (see e.g., \cite {CL09} and the references therein). Thus a dual problem is to first find a ``best partition'' of the data. Once this partition is obtained,  the associated optimal subspaces can be easily found. In any case, the search for an optimal partition or optimal subspaces usually involves heavy computations that dramatically increases with  the dimensionality of $\mathcal H$. Thus one important feature is to map the data into a lower dimensional space, and solve the transformed problem in  this lower dimensional space.  If the mapping is chosen appropriately, the original problem can be solved exactly or approximately using the solution of the transformed data.

In this article, we concentrate on the non-linear subspace modeling problem when the  model is a finite union of subspaces of   $\R^N$ of dimension $k<< N$.  Our goal is to find transformations from a high dimensional space to lower dimensional spaces with the aim of solving the subspace modeling problem using the low dimensional transformed data. We find the optimal data partition for the transformed data and use this  partition for the original data  to obtain the subspace model associated to this partition. We then estimate the error between the model thus found and the optimal subspaces model for the original data.

\section{Preliminaries}

Since one of our goals is to model a set of data by a union of subspaces, we first  provide a measure of how well a given set of data can be modeled by a union of subspaces.

We will assume in this article that the data belongs to the finite dimensional space $\R^{N}$.
There is no loss of generality in doing that, since
 it is easy to see that the subspaces of any optimal solution
belong to the span of the data, which is a  finite dimensional subspace of  our
(possible infinite dimensional) Hilbert space. (see \cite{ACHM07}, Lemma 4.2).
So we can assume that the initial Hilbert space is the span of the data.

\begin{definition}\label{def-1}

Given a set of vectors $\F=\{f_1,\dots,f_m\}$ in $\R^N$, a real number $\rho\geq0$ and positive integers $l,k < N$ we will say that the data $\F$ is $(l,k,\rho)$-sparse if there exist subspaces $V_1,\dots,V_l$ of $\R^N$ with $\dim(V_i)\leq k$ for $i=1,\dots,l$, such that
$$
e(\F,\{V_1,\dots,V_l\})=\sum_{i=1}^m \min_{1\leq j\leq l} d^2(f_i,V_j)\leq\rho,
$$
where $d$  stands for the euclidean distance in $\R^N.$

When $\F$ is $(l,k,0)$-sparse, we will simply say that $\F$ is $(l,k)$-sparse.

\end{definition}

Note that if  $\F$ is $(l,k)$-sparse, there exist $l$ subspaces $V_1,\dots,V_l$ of dimension at most $k$, such that
$$\F\subseteq \cup_{i=1}^l V_i.$$

For the general case $\rho>0$, the $(l,k,\rho)$-sparsity of the data implies that $\F$ can be
partitioned into a small number of subsets, in such a way that each subset belongs to or  is at no more than $\rho$-distance from a low dimensional  subspace.
The collection of these subspaces provides an optimal non-linear sparse model for the data.

Observe that if the data $\F$ is  $(l,k,\rho)$-sparse, a model which verifies Definition \ref{def-1} provides a dictionary of length not bigger
than $lk$ (and in most cases much smaller) in which our data can be represented  using at most $k$ atoms
with an error smaller than $\rho$.

More precisely,
let $\{V_1, \dots, V_l\}$ be a collection of subspaces which satisfies Definition \ref{def-1} and $D$ a  set of vectors from $ \bigcup_j V_j$ that is minimal with the property that    its span contains  $ \bigcup_j V_j$.
Then for each $f\in\F$ there exists $\Lambda \subset D$ with $\#\Lambda \leq k$ such that
$$
\|f-\sum_{g\in\Lambda} \alpha_g g\|_2^2 \leq \rho,  \qquad \text{ for some scalars }\alpha_g.
$$

In  \cite{ACM08} the authors studied the problem of finding, for each given set of pairs
$(l,k)$, the minimum $\rho$-sparsity value of the data. They also provided an algorithm for finding the optimal value of $\rho$,
as well as the optimal subspaces associated with $\rho$ and the corresponding optimal partition of the data. Specifically, denote by $\B$ the collection of {\em bundles}  of subspaces of $\R^N$,
$$\B=\{B=\{V_1,\dots,V_l\}\,:\,\dim(V_i)\leq k,\,\,i=1,...,l\},$$
 and for $\F=\{f_1,\dots,f_m\}$ a finite subset of $\R^N$, define
\begin{equation}\label{e-0}
e_0(\F):=\inf\{e(\F,B)\,:\,B\in \B\}.
\end{equation}
As a special case of a general theorem  in \cite{ACM08} we obtain the next theorem.

\begin{theorem}\label{B0}

Let  $\F=\{f_1,\dots,f_m\}$ be vectors in $\R^N$, and let $l$ and $k$ be given $(l<m,\; k<N)$, then there exists a bundle $B_0=\{V^0_1,\dots,V^0_l\}\in \B$ such that
\begin{equation}\label{e-min}
e(\F, B_0)=e_0(\F)=\inf\{e(\F,B)\,:\,B\in \B\}.
\end{equation}
Any bundle $B_0\in\B$ satisfying (\ref{e-min}) will be called an optimal bundle for $\F$.

\end{theorem}

The following relations between partitions of the indices $\{1,\dots,m\}$ and bundles will be relevant for our analysis.

We will denote by $\bold{\Pi}_l(\{1,\dots,m\})$ the set of all $l$-sequences $\S=\{S_1,\dots,S_l\}$ of subsets of $\{1,\dots,m\}$ satisfying the property that for all $1\leq i,j\leq l$,
$$\bigcup_{r=1}^lS_r=\{1,\dots,m\}\quad\text{and}\quad S_i\cap S_j=\emptyset \,\text{ for }\,i\neq j.$$

We want to emphasize that this definition does not exclude the case when some of the $S_i$ are the empty set. By abuse of notation, we will still call the elements of $\bold{\Pi}_l(\{1,\dots,m\})$  {\it partitions} of $\{1,\dots,m\}$.

\begin {definition}
\label{bunpart}Given a bundle $B=\{V_1,\dots,V_l\}\in \B$,  we can split the set $\{1,\dots,m\}$ into a partition  $\S=\{S_1,\dots,S_l\}\in\BPi_l(\{1,\dots,m\})$ with respect to that bundle, by grouping together into $S_i$ the indices of the vectors in $\F$ that are closer to a given subspace $V_i$ than to any other subspace $V_j$, $j\neq i$. Thus, the partitions generated by $B$ are defined by $\S=\{S_1,\dots,S_l\}\in\BPi_l(\{1,\dots,m\})$, where
$$j\in S_i\quad\text{if and only if}\quad d(f_j,V_i)\leq d(f_j,V_h),\quad\forall\,h=1,\dots,l.$$

\end{definition}

We can also associate  to a given partition $\S \in \BPi_l$  the  bundles in $\B$ as follows:

\begin{definition}
Given a partition $\S=\{S_1,\dots,S_l\} \in \BPi_l$,  a bundle $B =\{V_1,\dots,V_l\} \in \B$ is generated by $\S$ if and only if for every $i=1,\dots, l$,
$$ \sum_{j\in S_i} d^2(f_j,V_i) \leq  \sum_{j\in S_i} d^2(f_j,W) \text{ for all subspaces } W \text{ such that dim(}W)\le k.$$

\end{definition}

In this way, for  a given data set $\F$, every bundle has a set of associated partitions (those that are generated by the bundle) and every partition has a set of associated bundles (those that are generated by the partition). Note however, that
the fact that $\S$  is generated by  $B$ does not imply that $B$ is generated by $\S$, and vice versa. However, if $B_0$ is an optimal bundle that solves  the problem  for the data $\F$ as in Theorem \ref{B0}, then  in this case, the partition $\S_0$ generated by $B_0$ also generates $B_0$.
On the other hand not every pair $(B,\S)$  with this property produces the minimal error $e_0(\F).$

Here and subsequently, the partition $\S_0$ generated by the optimal bundle $B_0$ will be called an optimal partition for $\F$.

If $M$ is a set of data and $V$ is a subspace of $\R^N$, we will denote by $E(M,V)$ the mean square error of the data $M$ to the
subspace $V$, i.e.
 \begin{equation}\label{error-E}
 E(M,V) = \sum_{f\in M} d^2(f,V).
 \end{equation}

\section {Main results}

The problem of finding the optimal union of subspaces  that best models a given set of data $\F$ when the dimension of the ambient space $N$ is large is computationally expensive. When the dimension $k$  of the subspaces is considerably smaller than $N$, it is natural to map the data onto a lower-dimensional subspace, solve an associated problem in the lower dimensional space and  map the solution back into the original space. Specifically, given the data set $\F=\{f_1,\dots,f_m\}\subset\R^N$ which is $(l,k,\rho)$-sparse and a sampling matrix  $A\in \R^{r\times N}$, with $r<<N$,  find the optimal partition  of the sampled data $\F':=A(\F)=\{Af_1,\dots,Af_m\}\subset\R^r$, and use this partition to find an approximate solution to the optimal model for $\F$.

\subsection {Dimensionality reduction: The ideal case $\rho=0$}\label{rho=0}

In this section we will assume that the data $\F=\{f_1,\dots,f_m\}\subset\R^N$ is $(l,k)$-sparse, i.e., there exist $l$ subspaces of dimension at most $k$ such that $\F$ lies in the union of these subspaces. For this ideal case, we will show that we can always recover the optimal solution to the original problem from the optimal solution to the problem in the low dimensional space as long as the low dimensional space has dimension $r>k$.

We will begin with the proof that for any sampling matrix $A\in \R^{r\times N}$, the measurements $\F'=A(\F)$ are $(l,k)$-sparse in $\R^r$.

\begin{lemma}\label{lema-datos-exact}

Assume the data $\F=\{f_1,\dots,f_m\}\subset\R^N$ is $(l,k)$-sparse and let $A\in \R^{r\times N}$.
Then $\F':=A(\F)=\{Af_1,\dots,Af_m\}\subset\R^r$ is $(l,k)$-sparse.

\end{lemma}

\begin{proof}

Let $V_1^0,\dots,V_l^0$ be optimal spaces for $\F$. Since $$\dim(A(V_i^0))\leq \dim(V_i^0)\leq k\quad\forall\,1\leq i\leq l,$$
and $$\F'\subset\bigcup_{i=1}^lA(V_i^0),$$
it follows that $W:=\{A(V_1^0),\dots,A(V_l^0)\}$ is an optimal bundle for $\F'$ and $e(\F', W)=0$.

\end{proof}

Let $\F=\{f_1,\dots,f_m\}\subset\R^N$ be $(l,k)$-sparse and $A\in \R^{r\times N}$. By Lemma \ref{lema-datos-exact}, $\F'$ is $(l,k)$-sparse.
Thus, there exists an optimal partition
$\S=\{S_1,\dots,S_l\}$ for $\F'$ in $\BPi_l(\{1,\dots,m\})$, such that
$$\F'\subset \bigcup_{i=1}^l  W_i,$$
where $W_i:=\span\{Af_j\}_{j\in S_i}$ and $\dim(W_i)\leq k$. Note that $\{W_1,\dots, W_l\}$ is an optimal bundle for $\F'$.

We can define the bundle $B_{\S}=\{V_1,\dots,V_l\}$ by
\begin{equation}\label{bun-S}
V_i:=\span\{f_j\}_{j\in S_i},\quad\forall\, 1\leq i\leq l.
\end{equation}
Since $\S\in\BPi_l(\{1,\dots,m\})$, we have that
$$\F\subset \bigcup_{i=1}^l V_i.$$
 Thus, the bundle $B_{\S}$ will be optimal for $\F$ if
$\dim(V_i)\leq k,\,\,\forall\,1\leq i\leq l$. The above discussion suggests the following definition:

\begin{definition}\label{def-adm}

Let $\F=\{f_1,\dots,f_m\}\subset\R^N$ be $(l,k)$-sparse. We will call a matrix  $A\in \R^{r\times N}$ \textit{admissible} for $\F$ if for every optimal partition $\S$ for $\F'$, the bundle $B_{\S}$ defined by (\ref{bun-S})  is optimal for $\F$.

\end{definition}

The next proposition states that almost all $A\in \R^{r\times N}$ are admissible for $\F$.

The Lebesgue measure of a set $E\subset \R^q$ will be denoted by $|E|$.

\begin{proposition}\label{prop-adm-0}

Assume the data $\F=\{f_1,\dots,f_m\}\subset\R^N$ is $(l,k)$-sparse and let $r>k$. Then, almost all $A\in \R^{r\times N}$ are admissible for $\F$.

\end{proposition}

\begin{proof}

If a matrix $A\in \R^{r\times N}$ is not admissible, there exists an optimal partition $\S\in\BPi_l$ for $\F'$ such that the bundle $B_{\S}=\{V_1,\dots,V_l\}$ is not optimal for $\F$.

Let $\Di_k$ be the set of all the subspaces $V$ in $\R^N$ of dimension bigger than $k$, such that
$V=\span{\{f_j\}_{j\in S}}$ with $S\subset \{1,\dots,m\}.$

Thus, we have that the set of all the matrices of $ \R^{r\times N}$ which are not admissible for $\F$ is contained in the set
$$\bigcup_{V\in \Di_k}\{A\in \R^{r\times N}\,:\,\dim(A(V))\leq k\}.$$

Note that the set $\Di_k$ is finite, since there are finitely many subsets of $ \{1,\dots,m\}.$ Therefore, the proof of the proposition is complete by showing that for a fixed subspace $V\subset\R^N$, such that $\dim(V)>k$, it is true that
\begin{equation}\label{dim0-1}
|\{A\in \R^{r\times N}\,:\,\dim(A(V))\leq k\}|=0.
\end{equation}

Let then $V$ be a subspace such that $\dim(V)=t >k.$
Given $\{v_1,\dots,v_t\}$ a basis for $V$, by abuse of notation, we continue to write $V$ for the matrix in $\R^{N\times t}$ with vectors $v_i$ as columns.  Thus, proving (\ref{dim0-1}) is equivalent to proving that
\begin{equation}\label{dim0-2}
|\{A\in \R^{r\times N}\,:\,\text{rank}(AV)\leq k\}|=0.
\end{equation}

As $\min\{r,t\}>k$,  the set $\{A\in \R^{r\times N}\,:\,\text{rank}(AV)\leq k\}$ is included in
\begin{equation}\label{pol-ecu-0}
\{A\in \R^{r\times N}\,:\,\text{det}(V^*A^*AV)=0\}.
\end{equation}
Since $\text{det}(V^*A^*AV)$ is a non-trivial polynomial in the $r\times N$ coefficients of $A$, the set (\ref{pol-ecu-0}) has Lebesgue measure zero. Hence,  (\ref{dim0-2}) follows.

\end{proof}

\subsection{Dimensionality reduction: The non-ideal case $\rho>0$}
\label {NIC}
Even if a set of data is drawn from a union of subspaces, in practice it is often corrupted by noise. Thus, in general $\rho>0$, and our goal is to  estimate the error produced when we solve the associated problem in the lower dimensional space and map the solution back into the original space.

Intuitively, if $A \in \R^{r\times N}$ is an arbitrary matrix, the set $\F^{'} = A\F$ will preserve the original sparsity only if the matrix $A$ does not change the geometry of the data in an essential way. One can think that in the {\em ideal} case, since the data is sparse, it actually lies in an union of low dimensional subspaces (which is a very thin set in the ambient space).

However, when the data is not 0-sparse, but only $\rho$-sparse with $\rho >0$, the optimal subspaces plus the data do not lie in a thin set. This is the main obstacle in order to obtain an analogous result as in the ideal case.

Far from having the result that for {\em almost any} matrix $A$ the geometry of the data will be preserved, we have the Johnson-Lindenstrauss lemma, that guaranties - for a given data set - the existence of {\em one} such matrix $A$.

In what follows, we will use random matrices to obtain positive results for the $\rho > 0$ case. 

Let $(\Omega,\pr)$ be a probability measure space. Given $r, N\in\N$, a random matrix  $A_{\w}\in \R^{r\times N}$ is a matrix with entries  $(A_{\w})_{i,j}=a_{i,j}(\w)$, where $\{a_{i,j}\}$ are independent and identically distributed random variables for every $1\leq i\leq r$ and $1\leq j\leq N$.

\begin{definition}
We say that  a random matrix $A_{\w}\in \R^{r\times N}$ satisfies the concentration inequality if for every  $0< \e <1$, there exists $c_0=c_0(\e)>0$ (independent of $r,N$) such that for any $x\in\R^N$,
\begin{equation}\label{rm}
\pr\Big((1-\e)\|x\|_2^2\leq\|A_{\w}x\|_2^2\leq(1+\e)\|x\|_2^2\Big)\geq 1-2e^{-rc_0}
\end{equation}

\end{definition}

Such matrices are easy to come by as
the next proposition shows \cite{Ach03}.

\begin{proposition}\label{prop-gb}

Let $A_{\w}\in\R^{r\times N}$ be a random matrix whose entries are chosen independently from either $\mathcal{N}(0,\frac{1}{r})$ or $\{\frac{-1}{\sqrt{r}},\frac{1}{\sqrt{r}}\}$ Bernoulli.
Then  $A_{\w}$ satisfies (\ref{rm}) with $c_0(\e)=\frac{\e^2}{4}-\frac{\e^3}{6}$.
\end{proposition}

By using random matrices $A_{\omega}$ satisfying (\ref{rm}) to produce the lower dimensional data set $\F^{'}$, we will be able to recover with high probability an optimal partition for $\F$ using the optimal partition of $\F^{'}$.

Below we will state the main results of Section \ref{NIC} and we will give their proofs in Section \ref{proofs}.

Note that by Lemma \ref{lema-datos-exact}, if  $\F=\{f_1,\dots,f_m\}\subset\R^N$  is $(l,k,0)$-sparse, then $A_{\w}(\F)$ is $(l,k,0)$-sparse for all $\w\in\Omega$. The following proposition is a generalization of Lemma \ref{lema-datos-exact} to the case where $\F$ is $(l,k,\rho)$-sparse with $\rho>0$.

\begin{proposition}
\label {P36}
Assume the data $\F=\{f_1,\dots,f_m\}\subset\R^N$  is $(l,k,\rho)$-sparse with $\rho>0$. If $A_{\w}\in\R^{r\times N}$ is a random matrix which satisfies  (\ref{rm}), then $A_{\w}\F$ is $(l,k,(1+\e)\rho)$-sparse with probability at least $1-2me^{-rc_0}$.
\end{proposition}

Hence if  the data is mapped with a random matrix which satisfies the concentration inequality, then  with high probability, the sparsity of the transformed data is close to the sparsity of the original data.
Further, as the following theorem shows, we obtain  an estimation for the error between $\F$ and the bundle generated by the optimal partition for $\F'=A_{\w}\F$.

Note that, given a constant  $\alpha>0$, the scaled data $\alpha\F=\{\alpha f_1,\dots,\alpha f_m\}$ satisfies that  $e(\alpha\F,B)=\alpha^2e(\F,B)$ for any bundle $B$.  So, an optimal bundle for $\F$ is optimal for $\alpha\F$, and vice versa.  Therefore, we can assume that the data $\F=\{f_1,\dots,f_m\}$ is {\it normalized}, that is, the  matrix $M\in\R^{N\times m}$ which has the vectors $\{f_1,\dots,f_m\}$ as columns has unitary Frobenius norm. Recall that the Frobenius norm of a matrix $M\in\R^{N\times m}$ is defined by
\begin{equation}\label{fr-norm}
\|M\|^2:=\sum_{i=1}^{N}\sum_{j=1}^{m}M_{i,j}^2,
\end{equation}
where $M_{i,j}$ are the coefficients of $M$.

\begin{theorem}\label{theo-final}
\label {T37}
Let $\F=\{f_1,\dots,f_m\}\subset\R^N$ be a normalized data set and $0< \e <1$. Assume that $A_{\w}\in\R^{r\times N}$ is a random matrix  satisfying (\ref{rm}) and $\S_{\w}$ is an optimal partition for $\F'=A_{\w}\F$ in $\R^r$. If  $B_{\w}$ is a bundle generated by the partition $\S_{\w}$ and the data $\F$ in $\R^N$ as in Definition \ref {bunpart}, then with probability exceeding $1-(2 m^2+4m)e^{-rc_0}$, we have
\begin{equation}\label{teo-f}
e(\F,B_{\w})\leq (1+\e)e_0(\F)+\e c_1,
\end{equation}
where  $c_1=(l (d-k))^{1/2}$ and  $d=\rg(\F)$.

\end{theorem}

Finally, we can use this theorem to show that the set of matrices which are {\em $\eta$-admissible} (see definition below) is large.

The following definition generalizes Definition \ref{def-adm} to the $\rho$-sparse setting, with $\rho>0$.

\begin{definition}

Assume $\F=\{f_1,\dots,f_m\}\subset\R^N$ is $(l,k,\rho)$-sparse and let $0<\eta <1$. We will say that a matrix $A\in\R^{r\times N}$ is $\eta$-{\it admissible} for $\F$ if for any optimal partition $\S$  for $\F'=A\F$ in $\R^r$, the bundle $B_{\S}$ generated by $\S$ and $\F$ in $\R^N$, satisfies
$$e(\F,B_{\S})\leq \rho+\eta.$$

\end{definition}

We have the following generalization of Proposition \ref{prop-adm-0}, which provides an estimate on the size of the set of $\eta$-admissible matrices.

\begin{corollary}

Let $\F=\{f_1,\dots,f_m\}\subset\R^N$  be a normalized data set and $0<\eta<1$. Assume that $A_{\w}\in\R^{r\times N}$ is a random matrix which satisfies property (\ref{rm}) for  $\e= \eta\; (1+\sqrt{l (d-k)})^{-1}$.
Then $A_{\w}$ is $\eta$-admissible for $\F$ with probability at least $1-(2m^2+4m)e^{-r c_0(\e)}.$

\end{corollary}

\begin{proof}
Using the fact that $e_0(\F)\leq E(\F,\{0\})=\|\F\|^2=1$, we conclude from Theorem~\ref{T37} that
\begin{equation}\label{nota-f}
\pr\Big(e(\F,B_{\w})\leq e_0(\F)+ \e (1+c_1) \Big) \geq 1-c_2e^{-rc_0(\e)},
\end{equation}
where $c_1=(l (d-k))^{1/2}$, $d=\rg(\F)$, and $c_2=2 m^2+4m$. That is,
$$ \pr\Big(e(\F,B_{\w})\leq e_0(\F)+\eta \Big)\geq 1-(2m^2+4m)e^{-r c_0(\e)}.$$
\end{proof}
As a consequence of the previous corollary, we have a bound on the dimension of the lower dimensional space to obtain a bundle which produces an error at $\eta$-distance of the minimal error with high probability.

Now, using that $c_0(\e)\geq\frac{\e^2}{12}$ for random matrices with gaussian or Bernoulli entries (see Proposition \ref{prop-gb}), from Theorem~\ref{T37} we obtain the following corollary.

\begin{corollary}
Let $\eta, \dd \in (0,1)$, be given. Assume that $A_{\w}\in\R^{r\times N}$ is a random matrix whose entries are as in Proposition \ref{prop-gb}.

 Then for every $r$ satisfying,
 $$r\geq \frac{12 (1+\sqrt{ l (d-k)})^2}{\eta^2}\ln\Big(\frac{2m^2+4m}{\dd}\Big)$$
 \vspace{.08mm}
with probability at least  $1-\dd$ we have that

$$e(\F,B_{\w})\leq e_0(\F)+\eta.$$
\end{corollary}

We want to remark here that the results of subsection \ref{NIC} are valid for any
probability distribution that satisfies the concentration inequality (\ref{rm}).
The bound on the error is still valid for ${\rho =0}$.
However in that case we were able to obtain sharp results.

\section {Proofs} \label{proofs}

\subsection{Background and supporting results}

Before proving the results of the previous section we need several known theorems, lemmas, and propositions below.

  Given $M\in\R^{m\times m}$ a Hermitian matrix, let $\lambda_1(M)\geq\lambda_2(M)\geq\dots\geq\lambda_m(M)$ be its eigenvalues
and $s_1(M)\geq s_2(M)\geq\dots\geq s_m(M)\geq0$ be its singular values.

Recall that the Frobenius norm defined in (\ref{fr-norm}) satisfies that
$$\|M\|^2=\sum_{1\leq i,j\leq m}M_{i,j}^2=\sum_{i=1}^m s_i^2(M),$$
where $M_{i,j}$ are the coefficients of $M$.

Given $x\in\R^N$, we write $\|x\|_2$ for the $\ell^2$ norm of $x$ in $\R^N$.

\begin{theorem}\cite[Theorem III.4.1]{Bha97}\label{teo-Bha}

Let $A,B\in \R^{m\times m}$ be Hermitian matrices. Then for any choice of indices $1\leq i_1<i_2<\dots<i_k\leq m$,
$$\sum_{j=1}^k(\lambda_{i_j}(A)-\lambda_{i_j}(B))\leq\sum_{j=1}^k\lambda_j(A-B).$$

\end{theorem}

\begin{corollary}\label{cor-Bha}

Let $A,B\in \R^{m\times m}$ be Hermitian matrices. Assume $k$ and $d$ are two integers which satisfy $0\leq k\leq d\leq m$, then
$$\Big|\sum_{j=k+1}^d(\lambda_{j}(A)-\lambda_{j}(B))\Big|\leq(d-k)^{1/2}\|A-B\|.$$

\end{corollary}

\begin{proof}

Since $A-B$ is Hermitian, it follows that for each $1\leq j\leq m$ there exists $1\leq i_j\leq m$ such that
\begin{equation*}
|\la_j(A-B)|=s_{i_j}(A-B).
\end{equation*}

From this and Theorem \ref{teo-Bha} we have
\begin{eqnarray*}
\sum_{j=k+1}^d(\lambda_{j}(A)-\lambda_{j}(B))&\leq&\sum_{j=1}^{d-k}\lambda_j(A-B)\leq\sum_{j=1}^{d-k} s_{i_j}(A-B)\\
&\leq&\sum_{j=1}^{d-k} s_j(A-B)\leq(d-k)^{1/2}\Big(\sum_{j=1}^{d-k} s_j^2(A-B)\Big)^{1/2}\\
&\leq&(d-k)^{1/2}\|A-B\|.
\end{eqnarray*}

\end{proof}

\begin{remark}

Note that the bound of the previous corollary is sharp. Indeed, let $A\in\R^{m\times m}$ be the diagonal matrix with coefficients $a_{ii}=2$ for $1\leq i\leq d$, and $a_{ii}=0$ otherwise. Let $B\in\R^{m\times m}$ be the diagonal matrix with coefficients $b_{ii}=2$ for $1\leq i\leq k$, $b_{ii}=1$ for $k+1\leq i\leq d$, and $b_{ii}=0$ otherwise. Thus,
$$\Big|\sum_{j=k+1}^d(\lambda_{j}(A)-\lambda_{j}(B))\Big|=\Big|\sum_{j=k+1}^d(2-1)\Big|=d-k.$$
Further
$\|A-B\|=(d-k)^{1/2},$
and therefore
$$\Big|\sum_{j=k+1}^d(\lambda_{j}(A)-\lambda_{j}(B))\Big|=(d-k)^{1/2}\|A-B\|.$$
\end{remark}

\begin{lemma}\cite{AV06}\label{cotaY}
Suppose that $A_{\w}\in\R^{r\times N}$ is a random matrix which satisfies  (\ref{rm}) and $u,v\in\R^N$, then
$$|\langle u,v\rangle-\langle A_{\w}u,A_{\w}v\rangle|\leq \e\|u\|_2\|v\|_2,$$
with probability at least $1-4e^{-rc_0}.$

\end{lemma}

The following proposition was proved in \cite{Sar08}, but we include its proof for the sake of completeness.

\begin{proposition}\label{prop-norm}

Let $A_{\w}\in\R^{r\times N}$be a random matrix which satisfies  (\ref{rm}) and $M\in\R^{N\times m}$ be a matrix. Then, we have
$$\|M^*M-M^*A_{\w}^*A_{\w}M\|\leq \e\|M\|^2,$$
with probability at least $1-2(m^2+m)e^{-rc_0}.$

\end{proposition}

\begin{proof}

Set $Y_{i,j}(\w)=(M^*M-M^*A_{\w}^*A_{\w}M)_{i,j}=\langle f_i,f_j\rangle-\langle A_{\w}f_i,A_{\w}f_j\rangle.$ By Lemma \ref{cotaY} with probability at least $1-4e^{-rc_0}$ we have that
\begin{equation}\label{cotaY2}
|Y_{i,j}(\w)|\leq \e\|f_i\|_2\|f_j\|_2
\end{equation}

\smallskip

Note that if (\ref{cotaY2}) holds for all $1\leq i\leq j\leq m$, then

\begin{eqnarray*}
\|M^*M-M^*A_{\w}^*A_{\w}M\|^2&=&\sum_{1\leq i,j\leq m} Y_{i,j}(\w)^2 \\
&\leq& \e^2\sum_{1\leq i,j\leq m} \|f_i\|_2^2\|f_j\|_2^2=\e^2\|M\|^4.
\end{eqnarray*}

Thus, by the union bound, we obtain

\begin{eqnarray*}
\lefteqn{\pr\Big(\|M^*M-M^*A_{\w}^*A_{\w}M\|  \leq \e\|M\|^2\Big)         }\\
&\geq \pr\Big(|Y_{i,j}(\w)|\leq \e  \|f_i\|_2\|f_j\|_2\quad\forall\,1\leq i\leq j\leq m\Big)\\
&\geq 1-\sum_{1\leq i\leq j\leq m}4e^{-rc_0}=1-2(m^2+m)e^{-rc_0}.
\end{eqnarray*}

\end{proof}

\subsection{New results and proof of Theorem \ref{theo-final}}

Given $M\in\R^{N\times m}$ with columns $\{f_1,\dots,f_m\}$ and a subspace $V\subset\R^N$, let $E(M,V)$ be as in (\ref{error-E}), that is
$$E(M,V)=\sum_{i=1}^m d^2(f_i,V).$$
We will denote the $k$-minimal error associated with $M$ by
$$E_k(M):=\min_{V\,:\,\dim(V)\leq k}E(M,V).$$

Let $d:=\rg(M)$. Eckart-Young's Theorem (see \cite{Sch07}) states that
\begin{equation}\label{E-aut}
E_k(M)=\sum_{j=k+1}^d\la_j(M^*M),
\end{equation}
where $\la_1(M^*M)\geq\dots\geq\la_d(M^*M)>0$ are the positive eigenvalues of $M^*M$.

\begin{lemma}\label{lem-Ew}

Assume that $M\in\R^{N\times m}$  and $A \in\R^{r\times N}$ are arbitrary  matrices. Let $S \in\R^{N\times s}$ be a submatrix of $M$.  If $d:=\rg(M)$ is such that $0\leq k\leq d$, then
$$|E_k(S)-E_k(AS)|\leq (d-k)^{1/2}\,\|S^*S-S^*A^*AS\|.$$

\end{lemma}

\begin{proof}
Let $d_{s}:=\rg(S)$. We have $\rg(AS)\leq d_{s}$. If $d_s \leq k$, the result is trivial. Otherwise by (\ref{E-aut}) and Corollary \ref{cor-Bha}, we obtain
\begin{eqnarray*}
|E_k(S)-E_k(AS)|&=&\Big|\sum_{j=k+1}^{d_{s}}(\la_j(S^*S)-\la_j(S^*A^*AS))\Big|\\
&\leq&(d_{s}-k)^{1/2}\|S^*S-S^*A^*AS\|.
\end{eqnarray*}
As $S$ is a submatrix of $M$, we have that
$
(d_{s}-k)^{1/2}\leq(d-k)^{1/2},
$
which proves the lemma.

\end{proof}

Recall that $e_0(\F)$ is the optimal value for the data $\F$, and  $e_0(A_{\w}\F)$ is the optimal value for the data  $\F'=A_{\w}\F$ (See (\ref{e-0})). A relation between these two values is given by the following lemma.

\begin{lemma}\label{lem-e}

Let $\F=\{f_1,\dots,f_m\}\subset\R^N$ and $0<\e<1$. If $A_{\w}\in\R^{r\times N}$ is a random matrix which satisfies  (\ref{rm}), then with probability exceeding $1-2me^{-rc_0}$, we have
$$e_0(A_{\w}\F)\leq (1+\e)e_0(\F).$$

\end{lemma}

\begin{proof}

Let $V\subset\R^N$ be a subspace and $M\in\R^{N\times m}$ be a matrix. Using (\ref{rm}) and the union bound,
with probability at least $1-2m e^{-rc_0}$ we have that
\begin{eqnarray*}
E(A_{\w}M,A_{\w}V)&=&\sum_{i=1}^m d^2(A_{\w}f_i,A_{\w}V)\leq \sum_{i=1}^m \|A_{\w}f_i-A_{\w}(P_Vf_i)\|_2^2\\
&\leq& (1+\e)\sum_{i=1}^m \|f_i-P_Vf_i\|_2^2=(1+\e)E(M,V),
\end{eqnarray*}
where $P_V$ is the orthogonal projection onto $V$.

Assume that $\S=\{S_1,\dots,S_l\}$ is an optimal partition for $\F$ and $\{V_1,\dots,V_l\}$ is an optimal bundle for $\F$. Suppose that $m_i=\#(S_i)$ and $M_i\in\R^{N\times m_i}$ are the matrices which have $\{f_j\}_{j\in S_i}$ as columns. From what has been proved above and the union bound, with probability exceeding $1-\sum_{i=1}^l 2 m_ie^{-rc_0}=1-2m e^{-rc_0}$, it holds
$$
e_0(A_{\w}\F)\leq\sum_{i=1}^l E(A_{\w}M_i,A_{\w}V_i)\leq (1+\e)\sum_{i=1}^lE(M_i,V_i)=(1+\e)e_0(\F).
$$
\end{proof}
\begin {proof}[Proof of Proposition \ref {P36}] This is a direct consequence of Lemma \ref {lem-e}.
\end{proof}

\begin{proof}[Proof of Theorem \ref {T37}]
If $\S_{\w}=\{S_{\w}^1,\dots,S_{\w}^l\},$ and $m_{\w}^i=\#(S_{\w}^i)$, let $M_{\w}^i\in\R^{N\times m_{\w}^i}$ be the matrices which have $\{f_j\}_{j\in S_{\w}^i}$ as columns.
Since $B_{\w}=\{V_{\w}^1,\dots,V_{\w}^l\}$ is generated by $\S_{\w}$ and $\F$, it follows that $E(M_{\w}^i,V_{\w}^i)=E_k(M_{\w}^i)$.
And as $\S_{\w}$ is an optimal partition for $A_{\w}\F$ in $\R^r$, we have that $\sum_{i=1}^l E_k(A_{\w}M_{\w}^i)=e_0(A_{\w}\F)$.

Hence, using Lemma \ref{lem-Ew}, Lemma \ref{lem-e}, and Proposition \ref{prop-norm}, with high probability it holds that
\begin{eqnarray*}
e(\F,B_{\w})&\leq&\sum_{i=1}^lE(M_{\w}^i,V_{\w}^i)=\sum_{i=1}^l E_k(M_{\w}^i)\\
&\leq& \sum_{i=1}^l E_k(A_{\w}M_{\w}^i)+ (d-k)^{1/2}\sum_{i=1}^l \|M_{\w}^{i*}M_{\w}^i-M_{\w}^{i*}A_{\w}^*A_{\w}M_{\w}^i\|\\
&\leq&  e_0(A_{\w}\F)+ (l(d-k))^{1/2}\Big(\sum_{i=1}^l \|M_{\w}^{i*}M_{\w}^i-M_{\w}^{i*}A_{\w}^*A_{\w}M_{\w}^i\|^2\Big)^{1/2}\\
&\leq& (1+\e)e_0(\F)+ (l(d-k))^{1/2}\|M^{*}M-M^{*}A_{\w}^*A_{\w}M\|\\
&\leq&(1+\e)e_0(\F)+ \e (l (d-k))^{1/2},
\end{eqnarray*}
where $M\in\R^{N\times m}$ is the unitary Frobenius norm matrix which has the vectors $\{f_1,\dots,f_m\}$ as columns.

The right side of (\ref{teo-f}) follows from  Proposition \ref{prop-norm}, Lemma \ref{lem-e}, and the fact that
\begin{eqnarray*}
\lefteqn{\pr\Big(e(\F,B_{\w})\leq (1+\e)e_0(\F)+\e (l (d-k))^{1/2}\Big)      }    \\
&\geq& \pr\Big(\|M^{*}M-M^{*}A_{\w}^*A_{\w}M\|\leq \e \text{ and } e_0(A_{\w}\F)\leq(1+\e)e_0(\F)\Big)\\
&\geq& 1-(2 (m^2+m)e^{-rc_0}+2me^{-rc_0})= 1-(2 m^2+4m)e^{-rc_0}.
\end{eqnarray*}

\end{proof}

\section{Conclusions and related work}

The existence  of optimal union of subspaces models and an algorithm for finding them was obtained  in \cite{ACM08}.
In the present paper we have focused on the computational complexity of finding these models. More precisely,  we studied techniques of dimension reduction  for the algorithm proposed in \cite{ACM08}. These techniques can also be used in a wide variety of situations and are not limited to this particular application.

We  used random linear transformations to map the data to a lower dimensional space. The ``projected'' signals were then processed
in that space, (i.e. finding the optimal union of subspaces) in order to   produce an optimal partition.
Then we applied this partition to the original data to obtain the associated model for that partition and obtained a bound for the error.

We have analyzed two situations. First we studied the case
when  the data belongs to a union of subspaces (ideal case with no noise). In that case  we obtained the optimal model  using almost any transformation (see Proposition~\ref{prop-adm-0}).

In the presence of noise, the data usually doesn't belong to a union of low dimensional subspaces. Thus, the distances from the  data to an optimal model add up to a positive error.
In this case, we needed to restrict the admissible transformations. We applied recent results on distributions of matrices  satisfying concentration inequalities, which also proved to be very useful in the theory of compressed sensing.

We were able to prove that the model obtained by our approach is quasi optimal
with a high probability. That is, if we map the data using a random matrix from one of the distributions satisfying the concentration law, then with high probability, the distance of the data to the model is bounded by the optimal
distance plus a constant. This constant depends on the parameter of the concentration law, and the parameters of the model
(number and dimension of the  subspaces allowed in the model).

 Let us remark here that the problem of finding the optimal union of subspaces that fit a given data set
 is also known as ``Projective clustering". Several algorithms have been proposed in the literature to solve this problem.
Particularly relevant is \cite{DRVW06} (see also references therein) where the authors used results from volume and adaptive sampling to obtain a polynomial-time approximation scheme. See \cite{AM04} for a related algorithm.

\end{document}